\documentclass[graybox,envcountsec]{svmult}

\usepackage{mathptmx}       
\usepackage{helvet}         
\usepackage{courier}        
\usepackage{type1cm}        
\usepackage{makeidx}         
\usepackage{graphicx}        
\usepackage{multicol}        
\usepackage[bottom]{footmisc}

\usepackage{amsmath}
\usepackage{amssymb}
\usepackage{amscd}
\usepackage{indentfirst}
\usepackage{amsfonts}
\usepackage{dsfont}

\usepackage{mathptmx}
\DeclareMathAlphabet{\mathcal}{OMS}{cmsy}{m}{n}
\usepackage{mathtools} 

\usepackage{xcolor}

\usepackage{enumitem}

\newtheorem{thm}{Theorem}[section]
\newtheorem{cor}[thm]{Corollary}
\newtheorem{lem}[thm]{Lemma}
\newtheorem{prop}[thm]{Proposition}

\newtheorem{rem}[thm]{Remark}

\newcommand{\bt}{\begin{thm}}
\newcommand{\et}{\end{thm}}
\newcommand{\bl}{\begin{lem}}
\newcommand{\el}{\end{lem}}
\newcommand{\bc}{\begin{cor}}
\newcommand{\ec}{\end{cor}}
\newcommand{\br}{\begin{rem}}
\newcommand{\er}{\end{rem}}

\newcommand{\ba}{\begin{array}}
\newcommand{\ea}{\end{array}}
\newcommand{\bea}{\begin{eqnarray}}
\newcommand{\eea}{\end{eqnarray}}
\newcommand{\bead}{\begin{eqnarray*}}
\newcommand{\eead}{\end{eqnarray*}}
\newcommand{\be}{\begin{equation}}
\newcommand{\ee}{\end{equation}}
\newcommand{\bed}{\begin{displaymath}}
\newcommand{\eed}{\end{displaymath}}
\newcommand{\bps}{\begin{split}}
\newcommand{\eps}{\end{split}}
\newcommand{\la}{\label}

\newcommand\dR{{\mathbb{R}}}

\newcommand\dN{{\mathbb{N}}}

\newcommand\gotH{{\mathfrak{H}}}

\newcommand{\ga}{{\alpha}}
\newcommand{\gb}{{\beta}}

\newcommand{\gD}{{\Delta}}

\newcommand\cI{{\mathcal{I}}}

\newcommand\cL{{\mathcal{L}}}

\newcommand{\dom}{\mathrm{dom}}

\DeclareMathOperator*{\esssup}{ess\,sup}

\def\1{{\bf 1}}

\makeindex             

\begin{document}

\title*{Solution Operator for Non-Autonomous
Perturbation of Gibbs Semigroup
\\
\vspace{5mm}
}
\titlerunning{Non-Autonomous Dynamics}

\author{}
\authorrunning{V.~A.~Zagrebnov}

\maketitle

\vspace{-2cm}

\hspace {6cm} \textit{To the memory of Hagen Neidhardt}

\begin{center}

Valentin A. Zagrebnov

\smallskip

Aix-Marseille Universit\'{e}, CNRS, Centrale Marseille, I2M \\
Institut de Math\'{e}matiques de Marseille  (UMR 7373)\\
CMI - Technop\^{o}le Ch\^{a}teau-Gombert\\
39 rue F. Joliot Curie, 13453 Marseille, France
\end{center}

\vspace{1cm}

\abstract{\\
The paper is devoted to a linear dynamics for non-autonomous
perturbation of the Gibbs semigroup on a separable Hilbert space.
It is shown that {evolution family} $\{U(t,s)\}_{0 \leq s \leq t}$ solving the
non-autonomous Cauchy problem can be approximated in the \textit{trace-norm} topology by product
formulae. The rate of convergence of product formulae approximants
$\{U_n(t,s)\}_{\{0 \leq s \leq t, n\geq1\}}$ to the
{solution operator} $\{U(t,s)\}_{\{0 \leq s \leq t\}}$ is also established.
}
\numberwithin{equation}{section}
\renewcommand{\theequation}{\arabic{section}.\arabic{equation}}
\section{Introduction and main result} \label{sec:I}
\noindent
The aim of the paper is two-fold. Firstly, we study a linear dynamics,
which is a non-autonomous perturbation of Gibbs semigroup.
Secondly, we prove \textit{product formulae} approximations of the corresponding to this dynamics
\textit{solution operator} $\{U(t,s)\}_{\{0 \leq s \leq t\}}$, known also as \textit{evolution family},
\textit{fundamental solution}, or \textit{propagator}, see \cite{EngNag2000} Ch.VI, Sec.9.

To this end we consider on separable Hilbert space $\gotH$ a linear non-autonomous dynamics given by
evolution equation of the type:
\be\la{eq:1.1}
\begin{split}
\frac{\partial u(t)}{\partial t} = & -C(t) u(t), \quad u(s) = u_s,
\quad s \in [0,T)\subset \dR_{0}^{+},\\
C(t) := & \, A +B(t), \quad u_s \in \gotH ,
\end{split}
\qquad t \in \cI := [0,T],
\ee
where $\dR_{0}^{+} = \{0\} \cup \dR^{+}$ and linear operator $A$ is generator of a Gibbs semigroup.
Note that for the autonomous Cauchy problem (ACP) (\ref{eq:1.1}), when $B(t)=B$, the outlined programme
corresponds to the Trotter product formula approximation of the
Gibbs semigroup generated by a closure of operator $A + B$, \cite{Zag2019} Ch.5.

The main result of the present paper concerns the non-autonomous Cauchy problem (nACP) (\ref{eq:1.1})
under the following

\textbf{Assumptions:}

(A1) The operator $A \ge \mathds{1}$ in a separable Hilbert space $\gotH$ is self-adjoint.
The family $\{B(t)\}_{t \in \cI}$  of non-negative self-adjoint operators in $\gotH$ is such that
the bounded operator-valued function $(\mathds{1} + B(\cdot))^{-1}: \cI \longrightarrow \cL(\gotH)$
is strongly measurable.

(A2) There exists $\ga \in [0,1)$ such that inclusion: $\dom(A^\ga) \subseteq \dom(B(t))$, holds for a.e.
$t \in \cI$. Moreover, the function $B(\cdot)A^{-\ga}: \cI \longrightarrow \cL(\gotH)$
is strongly measurable and essentially bounded in the operator norm:
\be\la{eq:1.1a}
C_\ga := \esssup_{t\in \cI}\|B(t)A^{-\ga}\| < \infty.
\ee

(A3) The map $A^{-\ga}B(\cdot)A^{-\ga}: \cI \longrightarrow \cL(\gotH)$ is H\"older continuous in
the operator norm: for some $\gb \in (0,1]$ there is a constant $L_{\ga,\gb} > 0$  such that
one has estimate
\be\la{eq:1.1-1}
\|A^{-\ga}(B(t) - B(s))A^{-\ga}\| \le L_{\ga,\gb} |t-s|^\gb, \quad (t,s) \in \cI \times \cI .
\ee

(A4) The operator $A$ is generator of the Gibbs semigroup $\{G(t)= e^{-t A}\}_{t\geq 0}$, that is,
a strongly continuous semigroup such that $G(t)|_{t >0} \in \mathcal{C}_{1}(\mathfrak{H})$. Here
$\mathcal{C}_{1}(\mathfrak{H})$ denotes the $\ast$-ideal of \textit{trace-class} operators in
$C^*$-algebra $\mathcal{L}(\mathfrak{H})$ of bounded operators on $\mathfrak{H}$.

\br\la{rem:1.1-1}
Assumptions \emph{(A1)-(A3)} are introduced in \emph{\cite{IT1998}} to prove the
\emph{operator-norm} convergence of product formula approximants $\{U_n(t,s)\}_{0 \leq s \leq t}$
to solution operator $\{U(t,s)\}_{0 \leq s \leq t}$. Then they were widely used for product
formula approximations in \emph{\cite{NSZ2017}-\cite{Nickel2000}} in the context of
the \emph{evolution semigroup} approach to the nACP, see \emph{\cite{MR2000}-\cite{Nei1981}}.
\er
\br\la{rem:1.1} The following main facts were established
\emph{(}see, e.g., \emph{\cite{IT1998,NN2002,VWZ2009,Yagi1989}}\emph{)}
about the nACP for perturbed evolution equation of the type \eqref{eq:1.1}\emph{:}\\
\emph{(a)} By assumptions \emph{(A1)-(A2)} the operators $\{C(t) = A + B(t)\}_{t\in \cI}$ have
a common $\dom(C(t)) = \dom(A)$ and they are generators of contraction holomorphic semigroups. Hence,
the nACP \eqref{eq:1.1} is of \emph{parabolic} type \emph{\cite{Kato1961,Sob1961}}. \\
\emph{(b)} Since domains $\dom(C(t)) = \dom(A)$, $t\geq 0$, are dense, the nACP is \emph{well-posed}
with time-independent \emph{regularity} subspace $\dom(A)$.\\
\emph{(c)} Assumptions \emph{(A1)-(A3)} provide the existence of
\emph{evolution family} solving nACP \eqref{eq:1.1} which we call the \emph{solution operator}.
It is a strongly continuous, uniformly bounded family of operators
$\{U(t,s)\}_{(t,s) \in \gD}$, $\gD~:=~\{(t,s) \in \cI \times \cI: 0 \le s \le t\le T\}$,
such that the conditions
\be\la{eq:1.2}
\begin{split}
U(t,t) =& \ \mathds{1} \quad \mbox{for} \quad t \in \cI,\\
U(t,r)U(r,s) =& \ U(t,s), \quad \mbox{for}, \quad t,r,s \in \cI \quad \mbox{for}
\quad s \le r \le t,
\end{split}
\ee
are satisfied and $u(t) = U(t,s)\, u_s$ for any $u_s \in \mathfrak{H}_s$ is in a certain sense
\emph{(}e.g., \emph{classical, strict, mild}\emph{)} solution of the nACP \eqref{eq:1.1}.\\
\emph{(d)}
Here $\mathfrak{H}_s \subseteq \mathfrak{H}$ is an appropriate \emph{regularity subspace} of
initial data. Assumptions \emph{(A1)-(A3)} provide $\mathfrak{H}_s = \dom(A)$ and
$U(t,s)\mathfrak{H} \subseteq \dom(A)$ for $t > s$.
\er

In the present paper we essentially focus on convergence of the product \textit{approximants}
$\{U_n(t,s)\}_{(t,s)\in \gD, n\geq1}$ to {solution operator}
$\{U(t,s)\}_{(t,s)\in\gD}$. Let
\be\la{eq:1.6a}
s = t_1 < t_2 < \ldots < t_{n-1} < t_{n} < t, \quad t_k := s + (k-1)\tfrac{t-s}{n},
\ee
for $k \in \{1,2,\ldots,n\}$, $ n \in \dN$, be partition of the interval $[s,t]$. Then
corresponding approximants may be defined as follows:
\be\la{eq:1.6}
\begin{split}
W_{k}^{(n)}(t,s) :=& e^{-\tfrac{t-s}{n} A}e^{-\tfrac{t-s}{n} B(t_{k})}, \quad k = 1,2,\ldots,n,\\
U_n(t,s) :=& W_n^{(n)}(t,s)W_{n-1}^{(n)}(t,s)\times \cdots \times  W_2^{(n)}(t,s) W_1^{(n)}(t,s).\\
\end{split}
\ee

It turns out that if the assumptions (A1)-(A3), \textit{adapted} to a Banach space $\mathfrak{X}$,
are satisfied for $\alpha \in (0,1)$, $\beta \in (0,1)$ and in addition the condition
$\alpha < \beta$ holds, then solution operator $\{U(t,s)\}_{(t,s)\in \gD}$ admits the
operator-norm approximation
\be\la{eq:1.8v}
\esssup_{(t,s)\in \gD}\|U_n(t,s) - U(t,s)\| \le \frac{R_{\beta,\alpha}}{n^{\beta-\alpha}},
\quad n \in \dN,
\ee
for some constant $R_{\beta,\alpha}> 0$. This result shows that convergence of the
approximants $\{U_n(t,s)\}_{(t,s)\in \gD, n\geq1}$ is determined by the smoothness of the
perturbation $B(\cdot)$ in (A3) and by the parameter of inclusion in (A2), see \cite{NSZ2020}.

The {Lipschitz} case $\beta =1$ was considered for $\mathfrak{X}$ in \cite{NSZ2017}.
There it was shown that if $\alpha \in (1/2, 1)$, then one gets estimate
\be\la{eq:1.9-1}
\esssup_{t \in \cI}\|U_n(t,s) - U(t,s)\| \le \frac{R_{1,\alpha}}{n^{1-\alpha}}\ ,
\quad n=2,3,\ldots\,.
\ee

For the Lipschitz case in a Hilbert space $\mathfrak{H}$ the assumptions (A1)-(A3)
yield a stronger result \cite{IT1998}:
\be\la{eq:1.9}
\esssup_{(t,s)\in \gD}\|U_n(t,s) - U(t,s)\| \le R \ {{\frac{\log(n)}{n}}}, \quad n=2,3,\ldots\,.
\ee
Note that actually it is the best of known estimates for operator-norm rates of convergence
under conditions (A1)-(A3).

The estimate \eqref{eq:1.8v} was improved in \cite{NSZ2019} for $\alpha \in (1/2, 1)$ in a Hilbert space
using the \textit{evolution semigroup} approach \cite{Ev1976,How1974,Nei1981}. This approach is quite
different from technique used for (\ref{eq:1.9}) in \cite{IT1998}, but it is the same as that
employed in \cite{NSZ2017}.
\begin{prop}\la{prop:1.1}\emph{\cite{NSZ2019}}
Let assumptions \emph{(A1)-(A3)} be satisfied for $\beta \in (0,1)$.
If $\gb > 2\ga - 1 > 0$, then estimate
\be\la{eq:1.10}
\esssup_{(t,s)\in \gD}\|U_n(t,s) - U(t,s)\| \le \frac{R_{\beta}}{n^{\beta}}\, ,
\ee
holds for $n \in \dN$ and for some constant $R_\gb > 0$.
\end{prop}
Note that the condition $\gb > 2\ga-1$ is weaker than $\gb > \ga$ (\ref{eq:1.8v}), but
it does not cover the Lipschitz case (\ref{eq:1.9-1}) because of condition  $\beta < 1$.

{The main result of the present paper is the lifting of \textit{any} known operator-norm
bounds (\ref{eq:1.8v})-(\ref{eq:1.10}) (we denote them by
$R_{\alpha,\beta}\varepsilon_{\alpha,\beta}(n)$)
to estimate in the trace-norm topology $\|\cdot\|_1$. This is a subtle matter even for ACP, see
\cite{Zag2019} Ch.5.4.\\
- The first step is the construction for nACP \eqref{eq:1.1} a \textit{trace-norm} continuous solution
operator $\{U(t,s)\}_{(t,s)\in \gD}$, see Theorem \ref{thm:2.1} and Corollary \ref{cor:2.1}.\\
- Then in Section \ref{sec:III} for assumptions (A1)-(A4) we prove (Theorem \ref{thm:1.1}) the
corresponding \textit{trace-norm} estimate $R_{\alpha,\beta}(t,s)\varepsilon_{\alpha,\beta}(n)$ for
difference $\|U_n(t,s) - U(t,s)\|_1$.}
\begin{thm}\la{thm:1.1}
Let assumptions \emph{(A1)-(A4)} be satisfied. Then the estimate
\be\la{eq:1.11}
\|U_n(t,s) - U(t,s)\|_1 \le R_{\alpha,\beta}(t,s)\varepsilon_{\alpha,\beta}(n) \, ,
\ee
holds for $n \in \dN$ and $0\leq s < t \leq T$ for some $R_{\alpha,\beta}(t,s) > 0$.
\end{thm}
\section{Preliminaries}\la{sec:II}
\noindent
Besides Remark \ref{rem:1.1}(a)-(d) we also recall the following assertion, see, e.g., \cite{Sob1961},
Theorem 1, \cite{Tanabe1979}, Theorem 5.2.1.
\begin{prop}\la{prop:2.1}
Let assumptions \emph{(A1)-(A3)} be satisfied.\\
\emph{(a)} Then solution operator $\{U(t,s)\}_{(t,s)\in\gD}$ is strongly continuously differentiable
for $0 \leq s < t \leq T$ and
\begin{equation}\label{eq:2.1}
\partial_{t}U(t,s) = - (A + B(t))U(t,s).
\end{equation}
\emph{(b)} Moreover, the unique function $t \mapsto u(t) = U(t,s)\, u_s$ is a \emph{classical} solution
of \eqref{eq:1.1} for initial data  $\mathfrak{H}_s = \dom(A)$.
\end{prop}
Note that solution of \eqref{eq:1.1} is called \textit{classical} if
$u(t) \in C([0,T], \mathfrak{H})\cap C^1([0,T], \mathfrak{H})$, $u(t) \in \dom(C(t))$, $u(s) = u_s$,
and $C(t)u(t) \in C([0,T], \mathfrak{H})$ for all $t\geq s$, with convention that
$(\partial_{t}u)(s)$ is the right-derivative, see, e.g., \cite{Sob1961}, Theorem 1, or
\cite{EngNag2000}, Ch.VI.9.

Since the involved into {(A1), (A2)} operators are non-negative and self-adjoint,
equation (\ref{eq:2.1}) implies that the solution operator consists of \emph{contractions}:
\begin{equation}\label{eq:2.2}
\partial_{t}\|U(t,s) u\|^2 = - 2 (C(t)U(t,s)u,U(t,s)u) \leq 0, \ \ {\rm{for}}  \ \ u\in \mathfrak{H}.
\end{equation}

By (A1) $G(t) = e^{- t A}: \mathfrak{H} \rightarrow \dom(A)$. Applying to (\ref{eq:2.1})
the \textit{variation of constants} argument we obtain for $U(t,s)$ the integral equation :
\begin{equation}\label{eq:2.3}
U(t,s) = G(t-s) - \int_{s}^{t} d\tau \ G(t-\tau)\, B(\tau)\, U(\tau,s),
\quad U(s,s) = \mathds{1}\ .
\end{equation}
Hence evolution family $\{U(t,s)\}_{(t,s)\in\gD}$, which is  defined by equation (\ref{eq:2.3}),
can be considered as a \textit{mild} solution of nACP (\ref{eq:2.1}) for $0 \leq s \leq t \leq T$
in the Banach space $\mathcal{L}(\mathfrak{H})$ of bounded operators, cf. \cite{EngNag2000}, Ch.VI.7.

Note that assumptions {(A1)-(A4)} yield
for $0 \leq s < t \leq T$, $\tau \in (s,t)$ and for the closure $\overline{A^{-\alpha} B(\tau)}$:
\begin{equation}\label{eq:2.4}
\|G(t-s) A^{\alpha}\|_1 \leq \frac{M_\alpha}{(t - \tau)^{\alpha}} \ \ \
{\rm{and}} \ \ \ \|\overline{A^{-\alpha} B(\tau)}\| \leq C_\alpha \ .
\end{equation}
Then (\ref{eq:2.2}), (\ref{eq:2.4}) give the trace-norm estimate
\begin{equation}\label{eq:2.5}
\left\|\int_{s}^{t} d\tau \ G(t-\tau)\, B(\tau)\, U(\tau,s)\right\|_1 \leq
\frac{{M_\alpha}C_\alpha }{1- \alpha}{(t - s)^{1- \alpha}} \ ,
\end{equation}
and by (\ref{eq:2.3}) we ascertain that $\{U(t,s)\}_{(t,s)\in\gD} \in \mathcal{C}_{1}(\mathfrak{H})$
for $t>s$.

Therefore, we can construct solution operator $\{U(t,s)\}_{(t,s)\in\gD}$ as a \textit{trace-norm}
convergent Dyson-Phillips series  $\sum_{n=0}^\infty S_n(t,s)$ by {iteration} of the integral formula
(\ref{eq:2.3}) for $t>s$. To this aim we define the recurrence relation
\begin{equation}
\begin{split}
  &  S_{0}(t,s) = U_A(t-s), \\
  &  S_{n}(t,s) = -\int_{s}^t
  \!\mathrm{d}s\, G(t-\tau)\, B(\tau)\, S_{n-1}(\tau, s) , \quad n \geq 1.
\end{split}
  \label{3.49}
\end{equation}
Since in (\ref{3.49}) the operators $S_{n \geq 1}(t,s)$ are the $n$-fold
trace-norm convergent Bochner integrals
\begin{multline}\label{3.49-1}
  S_{n}(t,s)
  = \int_{s}^t \!\mathrm{d}\tau_1 \int_{s}^{\tau_1} \!\mathrm{d}\tau_2 \ldots
\int_{s}^{\tau_{n-1}} \!\mathrm{d}\tau_n \\
    G(t-\tau_1)(-B(\tau_1))G(\tau_1-\tau_2)\cdots
   G(\tau_{n-1}-\tau_n)(-B(\tau_n))G(\tau_n - s) ,
\end{multline}
by contraction property (\ref{eq:2.2}) and by estimate (\ref{eq:2.5}) there exit $0 \leq s \leq t$
such that ${{M_\alpha}C_\alpha }{(t - s)^{1- \alpha}}/{(1- \alpha)} =: \xi <1$ and
\begin{equation}\label{3.50-1}
\|S_{n}(t,s)\|_1 \leq \xi^n \ ,    \quad n \geq 1 .
\end{equation}
Consequently $\sum_{n=0}^\infty S_n(t,s)$ converges for $t>s$ in the trace-norm and satisfies the
integral equation (\ref{eq:2.3}). Thus we get for solution operator of nACP the representation
\begin{equation}
U(t,s)= \sum_{n=0}^\infty S_n(t,s) \, .
\label{3.48}
\end{equation}
This result can be extended to any $0 \leq s < t \leq T$ using (\ref{eq:1.2}).

We note that for $s \leq t$ the above arguments yield the proof of assertions in the next
Theorem \ref{thm:2.1} and Corollary \ref{cor:2.1}, but only in the \textit{strong}
(\cite{Yagi1989} Proposition 3.1, main Theorem in \cite{VWZ2009}) and in the
\textit{operator-norm} topology, \cite{IT1998} Lemma 2.1.
While for $t>s$ these arguments prove a generalisation of Theorem \ref{thm:2.1} and
Corollary \ref{cor:2.1} to the \textit{trace-norm} topology in  Banach space
$\mathcal{C}_{1}(\mathfrak{H})$:
\begin{thm}\label{thm:2.1}
Let assumptions \emph{(A1)-(A4)} be satisfied.
Then evolution family $\{U(t,s)\}_{(t,s)\in\gD}$ \emph{(\ref{3.48})}
gives for $t>s$ a mild trace-norm continuous solution of nACP \emph{(\ref{eq:2.1})} in
Banach space $\mathcal{C}_{1}(\mathfrak{H})$.
\end{thm}
\begin{cor}\label{cor:2.1}
For $t>s$ the evolution family $\{U(t,s)\}_{(t,s)\in\gD}$ \emph{(\ref{3.48})} is a strict solution
of the nACP \emph{:}
\be\la{eq:2.1a}
\begin{split}
\partial_{t} U(t,s) = & - C(t) U(t,s),
\ \ t \in (s,T)  \quad {\rm{and}}  \quad U(s,s) = \mathds{1},\\
C(t) := & \, A + B(t),
\end{split}
\quad (s,T) \subset [0,T],
\ee
in Banach space $\mathcal{C}_{1}(\mathfrak{H})$.
\end{cor}
\begin{proof}
Since by Remark \ref{rem:1.1}(c),(d) the function $t \mapsto U(t,s)$ for $t\geq s$ is strongly
continuous and since $U(t,s) \in \mathcal{C}_{1}(\mathfrak{H})$ for $t > s$, the product
$U(t+\delta,t)U(t,s)$ is continuous in the trace-norm topology for $|\delta|< t-s$ .
Moreover, since $\{u(t)\}_{s\leq t \leq T}$ is a classical solution of nACP (\ref{eq:1.1}),
equation (\ref{eq:2.1}) implies that $U(t,s)$ has strong derivative for any $t > s$. Then again
by Remark \ref{rem:1.1}(d) the trace-norm continuity of $\delta \mapsto U(t+\delta,t)U(t,s)$ and
by inclusion of ranges: ${\rm{ran}}(U(t,s))\subseteq \dom(A)$ for $t > s$, the trace-norm derivative
$\partial_{t} U(t,s)$ at $t(>s)$ exists and belongs to $\mathcal{C}_{1}(\mathfrak{H})$.

Therefore,
$U(t,s) \in C((s,T], \mathcal{C}_{1}(\mathfrak{H}))\cap C^1((s,T], \mathcal{C}_{1}(\mathfrak{H}))$
with $U(s,s) = \mathds{1}$ and $U(t,s)\in \mathcal{C}_{1}(\mathfrak{H})$,
$C(t)U(t,s) \in \mathcal{C}_{1}(\mathfrak{H})$ for $t>s$, which means that solution $U(t,s)$
of (\ref{eq:2.1a}) is \textit{strict}, cf. \cite{Yagi1989} Definition 1.1.
\hfill $\Box$ \end{proof}

We note that these results for ACP in Banach space $\mathcal{C}_{1}(\mathfrak{H})$ are well-known
for Gibbs semigroups, see \cite{Zag2019}, Chapter 4.

Now, to proceed with the proof of Theorem \ref{thm:1.1} about trace-norm convergence of the solution
operator approximants (\ref{eq:1.6}) we need the following preparatory lemma.
\begin{lem}\label{lem:2.1}
  Let self-adjoint positive operator $A$ be
  such that $e^{-t A} \in {\cal C}_1({\mathfrak{H}})$ for $t > 0$, and
  let $V_1,V_2,\ldots,V_n$ be bounded operators $\mathcal{L}(\mathfrak{H})$. Then
  \begin{equation}
    \Bigl\| \prod_{j=1}^n V_je^{-t_jA} \Bigr\|_1
    \leq
    \prod_{j=1}^n \|V_j\| \|e^{-(t_1+t_2+ \ldots +t_n)A/4}\|_1 \ ,
    \label{3.60}
  \end{equation}
for any set $\{t_1,t_2,\ldots,t_n\}$ of positive numbers.
\end{lem}
\begin{proof}
At first we prove this assertion for \textit{compact} operators: $V_j \in {\cal C}_\infty({\mathfrak{H}})$,
$j = 1,2,\ldots,n$. Let $t_m : = \min\{t_j\}_{j=1}^n >
0$ and $T : = \sum_{j=1}^n t_j > 0$. For any $1 \leq j \leq n$, we
define an integer $\ell_j \in \mathds{N}$ by
\begin{equation*}
  2^{\ell_j}t_m \leq t_j \leq 2^{\ell_j+1}t_m.
\end{equation*}
Then we get $\sum_{j=1}^n 2^{\ell_j}t_m > T/2$ and
\begin{equation}
\prod_{j=1}^n V_je^{-t_jA} = \prod_{j=1}^n V_j e^{-(t_j - 2^{\ell_j}t_m)A}(e^{-t_mA})^{\ell_j}.
\label{3.61}
\end{equation}
By the definition of the $\|\cdot\|_1$-norm and by inequalities
for singular values $\{s_k (\cdot)\}_{k\geq 1}$ of compact operators
\begin{align}
  \Bigl\| \prod_{j=1}^n V_je^{-t_jA} \Bigr\|_1
  &= \sum_{k=1}^\infty s_k\bigl( \prod_{j=1}^n V_je^{-(t_j -
  2^{\ell_j}t_m)A}(e^{-t_mA})^{2^{\ell_j}} \bigr)
\nonumber\\&
  \leq
  \sum_{k=1}^\infty \prod_{j=1}^n s_k\left(e^{-(t_j -
  2^{\ell_j}t_m)A}\right)
  \left[s_k(e^{-t_ma})\right]^{2^{\ell_j}}s_k(V_j)
\nonumber\\&
  \leq \sum_{k=1}^\infty s_k(e^{-t_mA})^{\sum_{j=1}^n2^{\ell_j}}
  \prod_{j=1}^n \|V_j\| \, .
\label{3.62}
\end{align}
Here we used that $s_k(e^{-(t_j - 2^{\ell_j}t_m)A}) \leq \|e^{-(t_j
- 2^{\ell_j}t_m)A}\| \leq 1$ and that $s_k(V_j) \leq \|V_j\|$. Let
$N : = \sum_{j=1}^n 2^{\ell_j}$ and $T_m : = Nt_m > T/2$. Since
\begin{equation*}
\sum_{k=1}^\infty s_k(e^{-t A/q})^{q} = (\|e^{-t A/q}\|_{q})^q \ ,
\end{equation*}
the inequality (\ref{3.62}) yields for $q = N$: 
\begin{equation}
 \Bigl\| \prod_{j=1}^n V_j e^{-t_jA} \Bigr\|_1 \leq
 \Bigl( \bigl\| e^{-T_mA/N} \bigr\|_{N} \Bigr)^N \prod_{j=1}^n \|V_j\|.
\label{3.63}
\end{equation}

Now we consider an integer $p \in \mathds{N}$ such that $2^p \leq N < 2^{p+1}$. It then follows
that $T/4 < T_m/2 < 2^pT_m/N$, and hence we obtain
\begin{align}\label{3.64}
 & \Bigl( \bigl\|e^{-T_mA/N} \bigr\|_{q=N} \Bigr)^N =
  \sum_{k=1}^\infty s_k^N(e^{-T_mA/N}) \\
& \leq  \sum_{k=1}^\infty s_k^{2^p}(e^{-2^pT_mA/2^pN}) \leq \sum_{k=1}^\infty
s_k^{2^p}(e^{-TA/2^{p+2}}) = \|e^{-TA/2^{2}}\|_1 \nonumber \ ,
\end{align}
where we used that $s_k(e^{-T_mA/N}) = s_k(e^{-2^pT_mA/2^pN}) \leq
\|e^{-T_mA/N}\| \leq 1$, and that $s_k(e^{-(t + \tau)A}) \leq
\|e^{-tA}\|s_k(e^{-\tau A}) \leq s_k(e^{-\tau A})$ for any $t,\tau
> 0$. Therefore, the estimates (\ref{3.63}), (\ref{3.64}) give the bound (\ref{3.60}).

Now, let $V_j \in {\cal L}({\mathfrak{H}})$, $j = 1,2,\ldots,n$,
and set $\tilde V_j : = V_je^{-\varepsilon A}$ for $0 <
\varepsilon < t_m$. Hence, $\tilde V_j \in {\cal
C}_1({\mathfrak{H}}) \subset \mathcal{C}_{\infty}(\mathfrak{H})$ and
$s_k(\tilde V_j) \leq \|\tilde V_j\| \leq \|V_j\|$. If we set $\tilde t_j : = t_j - \varepsilon$, then
\begin{equation}
  \Big\|\prod_{j=1}^n V_je^{-t_jA}\Big\|_1 \leq \prod_{j=1}^n
  \|V_j\| \|e^{-(\tilde t_1 + \tilde t_2 + \cdots + \tilde
  t_n) A/4}\|_1.
\label{3.65}
\end{equation}
Since the semigroup $\{e^{-t A}\}_{t \geq 0}$ is
$\|\cdot\|_1$-continuous for $t > 0$, we can take in
(\ref{3.65}) the limit $\varepsilon \downarrow 0$. This gives
the result (\ref{3.60}) in general case.
\hfill $\Box$ \end{proof}
\section{Proof of Theorem \ref{thm:1.1}}\la{sec:III}
\noindent
We follow the line of reasoning of the \textit{lifting} lemma developed in \cite{Zag2019}, Ch.5.4.1.

1. By virtue of (\ref{eq:1.2}) and (\ref{eq:1.6}) we obtain for difference in (\ref{eq:1.11})
formula:
\begin{equation}\label{eq:3.1}
U_n(t,s) - U(t,s) = \prod_{k=n}^{1}W_{k}^{(n)}(t,s) -
\prod_{l=n}^{1}U(t_{l+1} ,t_{l}) .
\end{equation}
Let integer $k_{n} \in (1, n)$. Then (\ref{eq:3.1}) yields the representation:
\begin{eqnarray*}
&&{{U_n(t,s) - U(t,s) =}}
\left(\prod_{k=n}^{k_n +1}W_{k}^{(n)}(t,s) - \prod_{l=n}^{k_n +1}U(t_{l+1} ,t_{l})\right)
\prod_{k=k_n}^{1}W_{k}^{(n)}(t,s) \nonumber \\
&&+ \prod_{l=n}^{k_n +1}U(t_{l+1} ,t_{l})\left(\prod_{k=k_n}^{1}W_{k}^{(n)}(t,s)-
\prod_{l=k_n}^{1}U(t_{l+1} ,t_{l})\right),
\nonumber
\end{eqnarray*}
which implies the trace-norm estimate
\begin{eqnarray}
&&\|U_n(t,s) - U(t,s)\|_1 \leq
\left\|\prod_{k=n}^{k_n +1}W_{k}^{(n)}(t,s) - \prod_{l=n}^{k_n +1}U(t_{l+1} ,t_{l})\right\|
\left\|\prod_{k=k_n}^{1}W_{k}^{(n)}(t,s)\right\|_1 \nonumber \\
&&+ \left\|\prod_{l=n}^{k_n +1}U(t_{l+1} ,t_{l})\right\|_1 \left\| \prod_{k=k_n}^{1}W_{k}^{(n)}(t,s)-
\prod_{l=k_n}^{1}U(t_{l+1} ,t_{l})\right\|. \la{eq:3.1a}
\end{eqnarray}

2. Now we assume that $\lim_{n \rightarrow \infty}k_{n}/n = 1/2 $. Then (\ref{eq:1.6a}) yields
$\lim_{n\rightarrow\infty} t_{k_{n}} = (t+s)/2$, $\lim_{n \rightarrow \infty} t_{{n}} = t$ and
uniform estimates (\ref{eq:1.8v})-(\ref{eq:1.10}) with the bound
$R_{\alpha,\beta}\varepsilon_{\alpha,\beta}(n)$ imply
\begin{eqnarray}\la{eq:3.2a}
&&\esssup_{(t,s)\in \gD}\left\|\prod_{k=n}^{k_n +1}W_{k}^{(n)}(t,s) -
U(t,(t+s)/2)\right\| \le R^{(1)}_{\alpha,\beta}\varepsilon_{\alpha,\beta}(n) , \\
&&\esssup_{(t,s)\in \gD}\left\| \prod_{k=k_n}^{1}W_{k}^{(n)}(t,s)-
U((t+s)/2, s)\right\| \le R^{(2)}_{\alpha,\beta}\varepsilon_{\alpha,\beta}(n) ,\la{eq:3.2b}
\end{eqnarray}
for $n \in \dN$ and for some constants $R^{(1,2)}_{\alpha,\beta} > 0$.

3. Since $\lim_{n \rightarrow \infty}k_{n}/n = 1/2 $ and $t > s$, by definition (\ref{eq:1.6}) and by
Lemma \ref{lem:2.1} for contractions $\{V_k = e^{-\tfrac{t-s}{n} B(t_{k})}\}_{k=1}^n$ there
exists $a_1 > 0$ such that
\begin{equation}\label{eq:3.3a}
\left\|\prod_{k=k_n}^{1}W_{k}^{(n)}(t,s)\right\|_1 =
\left\|\prod_{k=k_n}^{1} e^{-\tfrac{t-s}{n} A}e^{-\tfrac{t-s}{n} B(t_{k})}\right\|_1 \leq
a_1 \ \|e^{-\tfrac{t-s}{2} A}\|_1 \ .
\end{equation}
Similarly there is $a_2 > 0$ such that
\begin{equation}\label{eq:3.3b}
\left\|\prod_{l=n}^{k_n +1}U(t_{l+1} ,t_{l})\right\|_1 \leq a_2 \ \|e^{-\tfrac{t-s}{2} A}\|_1 \ .
\end{equation}

4. Since for $t > s$ the trace-norm $c(t-s): = \|e^{-\tfrac{t-s}{2} A}\|_1 < \infty$, by
(\ref{eq:3.1a})-(\ref{eq:3.3b}) we obtain the proof of the estimate (\ref{eq:1.11}) for
\begin{equation}\label{eq:3.4}
R_{\alpha,\beta}(t,s) := (a_1 R^{(1)}_{\alpha,\beta} + a_2 R^{(2)}_{\alpha,\beta}) \, c(t-s) \ ,
\end{equation}
and $0\leq s < t \leq T$. \hfill $\Box$
\begin{cor}\label{cor:3.1}
By virtue of Lemma \ref{lem:2.1} the proof of Theorem \ref{thm:1.1} can be carried over
almost verbatim for approximants $\{\widehat{U}_n(t,s)\}_{(t,s)\in \gD, n\geq1}$ \emph{:}
\be\la{eq:1.6-1}
\begin{split}
\widehat{W}_{k}^{(n)}(t,s) :=&e^{-\tfrac{t-s}{n} B(t_{k})}e^{-\tfrac{t-s}{n} A}, \quad k = 1,2,\ldots,n,\\
\widehat{U}_n(t,s) :=& \widehat{W}_n^{(n)}(t,s)\widehat{W}_{n-1}^{(n)}(t,s)
\times \cdots \times  \widehat{W}_2^{(n)}(t,s) \widehat{W}_1^{(n)}(t,s),\\
\end{split}
\ee
as well as for \textit{self-adjoint} approximants $\{\widetilde{U}_n(t,s)\}_{(t,s)\in \gD, n\geq1}$
\emph{:}
\be\la{eq:1.6-2}
\begin{split}
\widetilde{W}_{k}^{(n)}(t,s) :=&e^{-\tfrac{t-s}{n} A/2}e^{-\tfrac{t-s}{n} B(t_{k})}e^{-\tfrac{t-s}{n} A/2},
\quad k = 1,2,\ldots,n,\\
\widetilde{U}_n(t,s) :=& \widetilde{W}_n^{(n)}(t,s)\widetilde{W}_{n-1}^{(n)}(t,s)
\times \cdots \times  \widetilde{W}_2^{(n)}(t,s) \widetilde{W}_1^{(n)}(t,s).\\
\end{split}
\ee
For the both case the rate of convergence $\varepsilon_{\alpha,\beta}(n)$ for approximants
\emph{(\ref{eq:1.6-1}),(\ref{eq:1.6-2})} is the same as in \emph{(\ref{eq:1.11})}.
\end{cor}

Note that extension of Theorem \ref{thm:1.1} to Gibbs semigroups generated by a family
of non-negative self-adjoint operators $\{A(t)\}_{t \in \cI}$ can be done
along the arguments outlined in Section 2 of \cite{VWZ2009}. To this end one needs to add
more conditions to (A1)-(A4) that allow to control the family $\{A(t)\}_{t \in \cI}$.

Here we also comment that a general scheme of the \textit{lifting} due to Lemma \ref{lem:2.1} and
Theorem \ref{thm:1.1} can be applied to any symmetrically-normed ideal $\mathcal{C}_{\phi}(\mathfrak{H})$
of compact operators $\mathcal{C}_{\infty}(\mathfrak{H})$, \cite{Zag2019} Ch.6.
We return to this point elsewhere.

\section*{Acknowledgments}
This paper was motivated by my lecture at the Conference "Operator Theory and Krein Spaces"
(Technische Universit\"{a}te Wien, 19-22 December 2019), dedicated to the memory of
Hagen Neidhardt.

I am thankful to organisers: Jussi Behrndt, Aleksey Kostenko, Raphael Pruckner, Harald Woracek,
for invitation and hospitality.
\bibliographystyle{plain}
\def\cprime{$'$}

\end{document}